\documentclass[a4paper,12pt]{amsart}
\usepackage{amssymb}
\usepackage{amsmath}
\usepackage{stmaryrd}
\usepackage{amscd,amsthm,amssymb}
\usepackage{enumerate}% http://ctan.org/pkg/enumerate

\scrollmode
\usepackage{latexsym}

\def\.{\cdot}
\def\a{\alpha}

\def\vs{\vskip .6cm}

\def\la{\langle}
\def\ra{\rangle}

\def\beq{\begin{equation}}
\def\eeq{\end{equation}}
\def\bea{\begin{eqnarray*}}
\def\eea{\end{eqnarray*}}
\def\beaa{\begin{eqnarray}}
\def\eeaa{\end{eqnarray}}
\def\ba{\begin{array}}
\def\ea{\end{array}}
\def\f{\varphi}

\def \RM{\mathbb{R}}

\def \CM{\mathbb{C}}
\def \TM{\mathbb{T}}
\def \HM{\mathbb{H}}
\def \SM{\mathbb{S}}

\def\Ric{\mathrm{Ric}}
\def\id{\mathrm{id}}
\def\be{\begin{equation}}
\def\ee{\end{equation}}
\def\tr{\mathrm{tr}}

\def\so{\mathfrak{so}}

\def\SU{\mathrm{SU}}
\def\U{\mathrm{U}}

\def\G{\mathrm{G}}

\def\SO{\mathrm{SO}}

\def\End{\mathrm{End}}

\def\Cl{\mathrm{Cl}}
\def\Sp{\mathrm{Sp}}
\def\Spin{\mathrm{Spin}}

\def\Ker{\mathrm{Ker}}

\def\scal{\mathrm{scal}}

\def\T{\mathrm{T}}

\def\Z{{\mathcal Z}}

\newtheorem{epr}{Proposition}[section]
\newtheorem{ath}[epr]{Theorem}
\newtheorem{elem}[epr]{Lemma}
\newtheorem{ecor}[epr]{Corollary}

\theoremstyle{definition}

\title[Generalized Killing spinors]{Generalized Killing spinors on Einstein manifolds}

\author{Andrei Moroianu, Uwe Semmelmann}

\address{Andrei Moroianu \\ Universit\'e de Versailles-St Quentin \\
Laboratoire de Math\'ematiques \\ UMR 8100 du CNRS\\
45 avenue des \'Etats-Unis\\
78035 Versailles, France }
\email{am@math.polytechnique.fr}

\address{Uwe Semmelmann\\
Institut f\"ur Geometrie und Topologie \\
Fachbereich Mathematik\\
Universit{\"a}t Stuttgart\\
Pfaffenwaldring 57 \\
70569 Stuttgart, Germany
}
\email{uwe.semmelmann@mathematik.uni-stuttgart.de}

\date{\today}
\thanks{This work was partially supported by the contract ANR-10-BLAN 0105 ``Aspects
Conformes de la G{\'e}om{\'e}trie''. 
We thank Bernd Ammann for several suggestions of improvement, as well as Ilka Agricola, Diego
Conti, Vicente Cortes, Anna Fino and Thomas Friedrich for helpful comments.}

\begin{document}

\begin{abstract}
We study generalized Killing spinors on compact Einstein manifolds with positive scalar curvature.
This problem is related to the existence compact Einstein hypersurfaces in manifolds with parallel
spinors, or equivalently, in Riemannian products of flat spaces, Calabi-Yau, hyperk\"ahler, $\G_2$
and $\Spin(7)$ manifolds. 
\vs
\smallskip

\end{abstract}

\maketitle

\section{Introduction}

A {\it generalized Killing spinor} on a spin manifold $(M,g)$ is a non-zero spinor $ \Psi \in \Gamma(\Sigma M)$
satisfying for all vector fields $X$ the equation 
\beq\label{gks}
\nabla_X\Psi=A(X)\.\Psi,
\eeq 
where $A\in\Gamma(\End^+(\T M))$ is  some symmetric endomorphism field \cite{bgm,kf00,kf01}. 
If $A$ is a non-zero multiple of the identity, $\Psi$ is called a Killing spinor
\cite{ba,bfgk}. 

The interest in generalized Killing spinors is due to the fact that they arise in a natural way as
restrictions of parallel spinors to hypersurfaces. 
More precisely, if $(M^n,g)$ is a hypersurface of $(\Z^{n+1},g^\Z)$ and $\Phi$ is parallel spinor
on $\Z$, then its restriction to $M$ is a generalized Killing spinor with respect to the symmetric
tensor $A$ equal to half the second fundamental form of $M$, cf. \cite{bgm,friedrich:98}.
Conversely, if $\Psi$ is a generalized Killing spinor on $(M,g)$ with respect to $A$, then there
exists a metric on an open subset $\Z$ of $M\times \RM$ whose restriction to $M\times \{0\}$ is $g$
and a parallel spinor on $\Z$ whose restriction to $M\times \{0\}$ is $\Psi$ in the following cases:
\begin{enumerate}
\item If $A$ is a constant multiple of the identity, i.e. $\Psi$ is a Killing spinor \cite{ba};
\item Slightly more generally, if $A$ is parallel \cite{morel03};
\item Even more generally, if $A$ is a Codazzi tensor \cite{bgm};
\item In the generic case, under the sole assumption that $A$ and $g$ are analytic \cite{amm}.
\end{enumerate}

The common feature of the first three cases is that the ambient metric can be constructed
explicitly. In the last case, the existence of the ambient metric is given by the
Cauchy-Kovalevskaya theorem, and this explains the analyticity assumption. It is actually shown in
\cite{amm} that this assumption cannot be dropped.

Our main objective in this article is to study generalized Killing spinors on Einstein manifolds.
In some sense, this problem can be seen as an analogue of the Goldberg conjecture, which states that
an almost K\"ahler compact Einstein manifold with positive scalar curvature is K\"ahler (this 
conjecture was proved by Sekigawa \cite{seki}). 

In order to understand this analogy on needs to express both problems in terms of $G$-structures.
An almost Hermitian manifold is equivalent to a manifold with $\U(m)$-structure. The intrinsic
torsion of such a structure has 4 irreducible components (for $m\ge 3$) and 
being almost K\"ahler is equivalent to the vanishing of 3 components out of 4. The Goldberg
conjecture simply says that if the manifold is compact and Einstein with positive scalar curvature,
then the fourth component has to vanish too.

On the other hand, 
in small dimensions $n\le 8$ every (half-) spinor is {\em pure}, in the sense that the spin group
acts transitively on the unit sphere of the spin representation $\Sigma_n$ (or $\Sigma_n^\pm$ for
$n=4$ and $n=8$). Correspondingly, a non-vanishing (half-) spinor induces a $G$-structure on $M$
where $G$ (the stabilizer of a vector of the spin representation) equals $\Spin(7)\subset \SO(8)$,
$\G_2\subset \SO(7)$, $\SU(3)\subset \SO(6)$, $\SU(2)\subset \SO(5)$, $\SU(2)\subset \SO(4)$ and
$\{1\}\subset \SO(3)$ for $8\ge n\ge 3$ respectively. 
Being a generalized Killing spinor is
equivalent to the vanishing of certain components of the intrinsic
torsion of this $G$-structure.
More precisely, it is well known that the structure reduction defined by a generalized Killing
spinor is {\em co-calibrated} $\G_2$ (cf. \cite{cs06,fg}) for $n=7$, {\em half-flat} (cf.
\cite{chs,hi03}) for $n=6$ and {\em hypo} (cf. \cite{cs06}) for $n=5$. Note that for $n=4$ or $n=8$
a generalized Killing spinor $\Psi$ is never chiral (unless it is parallel), and each chiral part
$\Psi^+$ and $\Psi^-$ defines a structure reduction along the open set where it is non-vanishing.
The analogue of the Golberg conjecture in this setting (which turns out to be false in general, see
below) would be that a generalized Killing spinor on a compact Einstein manifold with positive
scalar curvature is necessarily Killing.

Note also that in this context, the embedding result in \cite{amm} for manifolds with generalized
Killing spinors
can be seen as a generalization to arbitrary dimensions of similar results by Bryant, Conti, Hitchin
and Salamon in small dimensions, cf. \cite{br, cs06, cs07, hi03}. 

Surprisingly, it turns out that the problem of finding all generalized Killing spinors on a given
spin manifold is out of reach at the present state of our knowledge. In dimension 2 already, the
fact that every generalized Killing spinor on $\SM^2$ is a Killing spinor, is non-trivial and
follows from Liebmann's theorem \cite{ku} (see Section \ref{dim2}).
Moreover, on the simplest Riemannian
manifold of dimension 3, the round 3-dimensional sphere, there is no classification available.
However, one can show that $\SM^3$ carries generalized Killing spinors which are not Killing spinors
(Section \ref{4.1} below).

In dimension 4, the analogue of the Goldberg conjecture holds. In Theorem \ref{dim4} below we show
that every generalized Killing spinor on a compact 4-dimensional Einstein manifold with positive
scalar curvature is Killing (and thus the manifold is isometric to $\SM^4$, cf.
\cite{ba,bfgk,hi86}). 

A similar result holds in dimension 5 for the round sphere (cf. Theorem \ref{dim5}). It
is presently unknown whether other 5-dimensional Einstein manifolds, e.g. the Riemannian product
$\SM^2(\tfrac1{\sqrt2})\times \SM^3$, carry generalized Killing spinors which are not Killing.

In dimensions 6 and 7 there are several examples of Einstein manifolds carrying generalized Killing
spinors which are not Killing. These examples correspond to half-flat structures on the Riemannian 
product $\SM^3\times\SM^3$ constructed by Schulte-Hengesbach \cite{sh}, who actually classified 
all left-invariant half-flat structures on $S^3\times S^3$  (see also Madsen and Salamon \cite{ms}) 
and to co-calibrated $\G_2$-structures on any 7-dimensional 3-Sasakian manifold,
including the sphere $\SM^7$,  constructed by Agricola and Friedrich \cite{af10} (see Section
\ref{4.4} below).

Finally, no examples of generalized Killing spinors on positive Einstein manifolds in dimension
$n\ge 8$ are known, other than Killing spinors on spheres, Einstein-Sasakian and 3-Sasakian
manifolds \cite{ba}.
We believe however that it should be possible to construct examples, at least in the 3-Sasakian
case,
using methods similar to those in \cite{af10}.

\section{Preliminaries}

For basic definitions and results on spin manifolds we refer to \cite{bfgk} and \cite{lm}.
Let $(M^n,g)$ be an $n$-dimensional Riemannian spin manifold with Levi-Civita connection $\nabla$.
The real spinor bundle $\Sigma M$ is endowed with a
connection, also denoted by $\nabla$, and a
Euclidean product $\la.,.\ra$ which is parallel with respect to $\nabla$:
$$\partial_X\la\Psi,\Phi\ra=\la\nabla_X \Psi,\Phi\ra+\la\Psi,\nabla_X\Phi\ra,\qquad\forall\ X\in\T
M,\ \Psi,\Phi\in\Gamma(\Sigma M).$$
The Clifford product with vectors is parallel with respect to $\nabla$ and skew-symmetric with
respect to $\la.,.\ra$, whence
\beq\label{xy}
\la X\cdot Y\cdot \Psi,\Psi\ra=-g(X,Y)\la\Psi,\Psi\ra,\qquad\forall\ X,Y\in\T M,\ \Psi\in\Sigma M.
\eeq

The Riemannian curvature $\mathcal R$ 
and the curvature $R^{\Sigma M}$ of the spinor bundle are related by 
\beq\label{curv-0}
R^{\Sigma M}_{X, Y}\, \Psi \;=\; \tfrac12 \mathcal R (X \wedge Y) \cdot \Psi \qquad\forall\ X,Y\in
\T M,\ \Psi\in\Sigma M ,
\eeq
where  $\mathcal R : \Lambda^2M \rightarrow \Lambda^2 M$ denotes the curvature operator 
defined by
$$
g( \mathcal R (X \wedge Y), U \wedge V ) = g( R_{X, Y}U, V).
$$ 
(In particular, the curvature operator on the standard sphere is minus the identity). Recall that
the Clifford multiplication with $2$-forms is defined via the equation
\beq\label{2f}
(X \wedge Y) \cdot \Psi \;=\; X \cdot Y \cdot \Psi \;+\; g(X, Y ) \, \Psi .
\eeq

Throughout this article we will identify $1$-forms and bilinear forms with vectors and
endomorphisms respectively, by the help of the Riemannian metric. In particular it makes sense to
speak about (skew)-symmetric endomorphism fields. The corresponding spaces will be denoted by
$$
\End^\pm(\T M)_p:=\{A\in\End(\T M)_p,\ |\ g(AX,Y)=\pm g(X,AY)\ \forall\ X,Y\in\T_p M\}.
$$
If $A\in \Gamma(\End^+(\T M))$ and $ \{ e_i \} $ is a local orthonormal frame, then 
\beq\label{trace}
\sum_{i=1}^n e_i \cdot A(e_i) \cdot \Psi=-\tr(A)\,\Psi.
\eeq

Applying the
first Bianchi identity the curvature relation  \eqref{curv-0} yields the well-known formula (see also \cite{bfgk}):
\beq\label{ricci}
\Ric(X) \cdot  \Psi \;=\; - 2 \sum_{i=1}^n e_i \cdot R^{\Sigma M} _{X, e_i} \, \Psi
\eeq
which together with \eqref{trace} yields:
\beq\label{scal}
\scal\, \Psi \;=\; - \sum_{i=1}^n e_i \cdot \Ric(e_i) \cdot \Psi .
\eeq

\section{Generalized Killing spinors}

Consider now a  generalized Killing spinor $\Psi$ on $(M,g)$, i.e. a spinor satisfying  the
equation $\nabla_X\Psi=A(X)\cdot\Psi$ for some symmetric endomorphism field $A$. Taking the scalar
product with 
$\Psi$ in this equation shows that the norm of $\Psi$ is constant. By rescaling, we may assume that $|\Psi|^2=1$.
Using \eqref{gks} together with \eqref{trace} shows that 
$$
D\Psi=-\tr(A)\Psi,
$$
where $D$ denotes the Dirac operator. We thus get
$$
D^2\Psi=\tr^2(A)\Psi-d\tr(A)\cdot\Psi.
$$
Moreover, taking a further covariant derivative in \eqref{gks} yields
$$
\nabla^*\nabla\Psi=-\sum_{i=1}^n (\nabla_{e_i} A)e_i\cdot\Psi-A(e_i)\cdot
A(e_i)\cdot\Psi=-\sum_{i=1}^n (\nabla_{e_i} A)e_i\cdot\Psi+\tr(A^2)\Psi,
$$
so the Lichnerowicz formula implies
\beq\label{sc}
\tfrac14\scal\Psi=D^2\Psi-\nabla^*\nabla\Psi=\tr^2(A)\Psi-d\tr(A)\cdot\Psi+\sum_{i=1}^n
(\nabla_{e_i} A)e_i\cdot\Psi-\tr(A^2)\Psi.
\eeq

Let us extend the action of $A$ to $2$-forms by
$A(X \wedge Y) = A(X)  \wedge A(Y)$. Using \eqref{2f}, the generalized Killing equation~\eqref{gks},
and the curvature relation~\eqref{curv-0}, it follows that 
\beq\label{curv}
\tfrac12 \, \mathcal R (X  \wedge Y) \cdot \Psi\;=\; 
R^{\Sigma M}_{X, Y} \, \Psi 
 \;=\; [(\nabla_X A)Y  \,-\, (\nabla_Y A)X  ] \cdot \Psi \;-\, 2\, A(X \wedge Y) \cdot \Psi.
\eeq

\begin{elem}
If $a$ denotes the trace of $A$ and $\delta A: = - \sum_{i=1}^n (\nabla_{e_i} A)e_i$ denotes the
divergence of $A$, then the following relations hold.
\medskip
\beq
\sum_{i=1}^n e_i \wedge (\nabla_{e_i} A)X \cdot \Psi  \;=\; [\tfrac12 \Ric(X) + 2 A^2(X) - 2a A(X)]
\cdot \Psi ,
\label{two}
\eeq
\beq
0   \;=\; \delta A \; +  \; d a  ,  \label{three1}
\eeq
\beq
\scal  \;=\; 4 a^2 \;-\; 4 \tr A^2 . \label{three2}
\eeq
\end{elem}
\proof
From \eqref{2f}, \eqref{trace}, \eqref{ricci} and \eqref{curv} we get for every tangent vector $X$:
\bea -\tfrac12\Ric(X)\cdot\Psi
&=&\sum_{i=1}^n e_i \cdot R^{\Sigma M} _{X, e_i} \, \Psi=\sum_{i=1}^n e_i\cdot [(\nabla_X A)e_i 
\,-\,  (\nabla_{e_i} A)X  ] \cdot \Psi \\
&&-\, 2 \sum_{i=1}^n e_i\cdot [A(X) \cdot A(e_i) +g(A(X),A(e_i))]\cdot \Psi\\
&=&-\tr(\nabla_XA)\Psi-\sum_{i=1}^n e_i\cdot (\nabla_{e_i} A)X\cdot\Psi\\&&+\,4
A^2(X)\cdot\Psi-2A(X)\cdot a\Psi-2A^2(X)\cdot\Psi\\
&=&[-da(X)-\delta A(X)]\Psi-\sum_{i=1}^n e_i \wedge (\nabla_{e_i} A)X \cdot \Psi +[2 A^2(X) - 2a
A(X)] \cdot \Psi.
\eea
Taking the scalar product with $\Psi$ in this equation and using the fact that the Clifford product
with 1- and 2-forms is skew-symmetric yields 
\eqref{three1}, and reinjecting in the same equation gives \eqref{two}. Finally, \eqref{three2}
follows from \eqref{sc} and \eqref{three1}.
\qed

\medskip

In order to rewrite the right hand side of \eqref{two} we introduce the symmetric endomorphism
$$
B\;:=\; A^2 \;-\; a A\;+\; \tfrac14  \Ric  .
$$
 Note that $B$ is traceless because of \eqref{three2} and $B$ vanishes if $A$ is a multiple of the
identity. We introduce the notation
$$
T^Z \;=\; \sum_{i=1}^n e_i \wedge (\nabla_{e_i} A) Z
\quad
\mbox{and}
\quad
T \;=\; \sum_{i=1}^n  T^{e_i} \otimes e_i  .
$$
Then $T^Z$ is a $2$-form on $M$ and, considering $A$ as a $1$-form on $M$ with values in $\T M$, we have
\beq\label{da1}
T^Z(X,Y) = g((\nabla_X A)Y  \,-\, (\nabla_Y A)X,Z) =  g((d^\nabla A)(X, Y),Z)  .
\eeq
The tensor  $T=d^\nabla A$ can also be considered as a map $T : \Lambda^2 M \rightarrow \T M$ by defining 
$$
g(T(X \wedge Y),Z) = T^Z(X, Y).
$$
Let $\sigma$ be an arbitrary $2$-form and $Z$ any vector field on $M$. Then  \eqref{curv} and
\eqref{two} can be rewritten as
\begin{align}
T(\sigma ) \cdot  \Psi  \quad & = \quad  [\, \tfrac12 \,  \mathcal{R} (\sigma) \, \;+ \; 2\,
A(\sigma) \, ] \cdot \Psi 
\label{curv-2}, &\forall\ \sigma\in\Lambda^2M \\[1.5ex]
T^Z \cdot \Psi  \quad & = \quad  2\, B(Z) \cdot \Psi & \forall\ Z\in \T M. \label{two-2}
\end{align}

\bigskip

\section{Generalized Killing spinors on low-dimensional Einstein manifolds}

\subsection{The case of dimension $2$} \label{dim2}
Any $2$-dimensional Einstein spin manifold of positive scalar curvature is homothetic to
the round sphere $\SM^2$. Using classical rigidity results it it easy to show that every generalized
Killing spinor on $\SM^2$ is a Killing spinor. 

Indeed, if $\Psi$ is a spinor satisfying
\eqref{gks} then every point of $\SM^2$ has a neighbourhood which embeds isometrically in $\RM^3$
with second fundamental form $2A$. In particular, the determinant of $A$ is constant equal to
$\tfrac14$ by Gauss' Theorema Egregium. If $A$ is not equal to $-\tfrac\id 2$, then there exists a
point in  $\SM^2$ where one of its eigenvalues attains its maximum which is strictly larger than
$\tfrac12$ and where the other eigenvalue attains its minimum, which is strictly smaller than
$\tfrac12$.

On the other hand, Liebmann's Theorem \cite{ku} states that
if at a non-umbilic point of a surface $S$ in $\RM^3$ one of the principal curvatures has a local
maximum and the other one has a local minimum, then the Gaussian curvature of $S$ is non-positive
at that point. This contradiction shows that $A$ has to be scalar.

\subsection{The case of dimension $3$} \label{4.1}
Any $3$-dimensional  Einstein manifold of positive scalar curvature is locally homothetic to the
round sphere $\SM^3$. We will show that $\SM^3$ carries generalized Killing spinors which are not
Killing. 

Recall that in dimension 3 the Clifford action of the volume form on the spin bundle is the
identity \cite{lm}. This readily implies 
\beq\label{v3}
\omega\cdot\Psi=-*\omega\cdot\Psi,\qquad\forall \ \omega\in\Lambda^2\SM^3,\ \Psi\in\Sigma \SM^3.
\eeq

It is well-known that $\SM^3$ carries an orthonormal frame of left-invariant Killing vector fields
$\{\xi_1,\xi_2,\xi_3\}$ satisfying
\beq\label{kvf1}
\nabla_{\xi_1}\xi_2=-\nabla_{\xi_2}\xi_1=\xi_3,\qquad
\nabla_{\xi_2}\xi_3=-\nabla_{\xi_3}\xi_2=\xi_1,\qquad
\nabla_{\xi_3}\xi_1=-\nabla_{\xi_1}\xi_3=\xi_2.
\eeq
One can express this in a more concise way by saying that any left-invariant vector field $\xi$ on
the Lie group $\SM^3$ satisfies
\beq\label{kvf}
\nabla_{X}\xi=*(X\wedge\xi),\qquad\forall\ X\in \T\SM^3.
\eeq

Let $\Phi$ be a Killing vector field on $\SM^3$ with Killing constant $\tfrac12$:
\beq\label{phi} \nabla_X\Phi=\tfrac12 X\cdot\Phi,\qquad\forall\ X\in \T\SM^3,
\eeq
and let $\xi$ be a unit left-invariant Killing vector field.
Using \eqref{2f}, \eqref{v3} and \eqref{kvf} we compute the covariant derivative of the spinor
$\Psi:=\xi\cdot\Phi$:
\bea
\nabla_X\Psi&=&(\nabla_X\xi)\cdot\Phi+\tfrac12\xi\cdot X\cdot\Phi\\
&=&-(X\wedge\xi)\cdot\Phi-\tfrac12X\cdot\xi\cdot\Phi-g(X,\xi)\Phi\\
&=&-X\cdot\xi\cdot\Phi-g(X,\xi)\Phi-\tfrac12X\cdot\xi\cdot\Phi-g(X,\xi)\Phi\\
&=&-\tfrac32X\cdot\Psi+2g(X,\xi)\xi\cdot\Psi.
\eea
This shows that $\Psi$ is a generalized Killing spinor corresponding to the symmetric endomorphism field
$$X\mapsto A(X):=-\tfrac32X+2g(X,\xi)\xi.$$

As a matter of fact, note that $A$ is not a Codazzi tensor.
\medskip 

\subsection{The case of dimension $4$}
We assume  that $\Psi$ is a generalized Killing spinor on a compact oriented $4$-dimensional
Einstein manifold $(M,g)$ of positive scalar curvature. We thus have $\Ric = \lambda g$, with
$\lambda = \tfrac{\scal}{4} > 0$, and like before we may assume that $\Psi$ is scaled to have unit
length.  Then  \eqref{three2} reads $a^2 - \tr A^2 = \lambda$.

In dimension $4$ the spin representation splits as $\Sigma = \Sigma^+ \oplus \Sigma^-$, where
$\Sigma^\pm $ are the $\pm 1$-eigenspaces of 
the multiplication with the volume element and are interchanged by Clifford multiplication with
vectors. Correspondingly,
$\Psi$ splits as $\Psi \;=\; \Psi^+ \;+\; \Psi^-$ with
\beq\label{pm}
\nabla_X \Psi^\pm \;=\; A(X) \cdot \Psi^\mp .
\eeq

Let $M_0$ denote the open set $p \in M$ with $\Psi^-_p \neq 0$. 
We claim that $M_0$ is dense. Indeed, if $U$ were a non-empty
open subset of $M\setminus M_0$, then \eqref{pm} yields $A(X)\cdot \Psi^+=0$ for all $X\in\T U$, so
$A|_U=0$. By \eqref{pm} again, $\Psi^+$ is parallel on $U$, so the Ricci tensor vanishes on $U$,
contradicting the fact that $\scal>0$.

Let $h : =|\Psi^-|^2$ be the length function of $\Psi^-$ and let  $\eta$ be the vector field on $M$ given by 
\beq\label{eta}g(\eta,X) = \la X\cdot\Psi^+,\Psi^-\ra,\qquad\forall \ X\in\T M.\eeq
For every $p\in M_0$ the injective map $X \in \T_p M \mapsto X \cdot \Psi^- \in \Sigma^+_pM$ is
bijective since $\dim \T _pM = \dim \Sigma^+_pM$.
Let $\xi$ denote the vector field on $M_0$ defined by $\Psi^+ = \xi \cdot \Psi^-$. Using \eqref{xy} we get 
$\eta(X)=-h\,g(X,\xi).$
Moreover, since $1 = |\Psi|^2 =  |\Psi^+|^2  +  |\Psi^-|^2 = |\Psi^-|^2 (1 + |\xi|^2 )$ 
we infer $|\xi|^2 = \tfrac{1}{h}-1$ and thus
\beq\label{et}|\eta|^2 = h - h^2 .\eeq

\begin{elem}\label{dh}
\begin{enumerate}[(i)]
\item $ \quad dh \;=\; 2\, A(\eta )  $
\item  $\quad  \nabla_X \eta \;=\; (1 - 2h)\, A(X) $
\item $\quad d \eta \;=\; 0, \quad \delta \eta \;=\; -\,(1 - 2h) \, a $
\end{enumerate}
\end{elem}

\proof
{\it (i)} Using \eqref{pm} we compute for every $X\in\T M$:
$$
d |\Psi^-|^2(X) =
2\la \nabla_X \Psi^-, \Psi^-\ra \;=\; 2\la  A(X) \cdot \Psi^+, \Psi^- \ra \;=\; 2\eta(A(X))=2g(A(\eta),X).
$$

{\it (ii)}
Taking the covariant derivative in the direction of $Y$ in \eqref{eta}, assuming that $X$ is
parallel at some point and using \eqref{xy} and \eqref{pm} yields
\bea
g( \nabla_Y \eta, X) & =& \la X \cdot A(Y) \cdot \Psi^-, \Psi^- \ra   \;+\; \la X \cdot \Psi^+, A(Y) \cdot \Psi^+ \ra  \\
& = &- g( X, A(Y)) \, |\Psi^-|^2 \;+\; g (X, A(Y)) \,  |\Psi^+|^2   \\
& = & (1-2h) g( A(Y), X) .
\eea
{\it (iii)} Follows immediately from {\it (ii)}.
\qed

\bigskip

\begin{ecor}\label{Delta}
$ \quad \Delta h \;=\; - \, 2\,  da(\eta) \;-\; 2 \, (a^2 - \lambda) \,(1 - 2h) $
\end{ecor}
\proof
Straightforward calculation using \eqref{three1}:
\bea
\Delta h & =& \delta d h \;=\;2 \,  \delta (A (\eta)) \;=\; - 2\sum_{i=1}^n g( e_i, (\nabla_{e_i}
A) \eta + A(\nabla_{e_i} \eta)) \\
& =& 2 g( \delta A, \eta )\;-\; 2 \sum_{i=1}^n g( \nabla_{e_i} \eta, A(e_i)) \\[1ex]
& =& - 2 da( \eta) \;-\; 2 (1-2h) \tr (A^2)
\;=\;  - 2 da( \eta) \;-\; 2 (1-2h) (a^2 - \lambda) .
\eea
\qed

We denote by $M_1$ the set of points where $\Psi^+$ is non-vanishing. Like before, $M_1$ is dense,
so $M':=M_0\cap M_1$ is dense, too.

It is well known that $\Lambda^2_\pm M$ acts trivially (by Clifford multiplication) on $\Sigma^\pm
M$. Moreover, the map 
$
\omega^+ \mapsto \omega^+ \cdot \Psi^-
$ is a bijection from the space of self-dual $2$-forms
$\Lambda^2_+M'$ onto the orthogonal complement $(\Psi^-)^\perp $  in $ \Sigma^-M'$, and similarly the map 
$
\omega^- \mapsto \omega^- \cdot \Psi^+
$ is a bijection from the space of anti-self-dual $2$-forms
$\Lambda^2_-M'$ onto the orthogonal complement $(\Psi^+)^\perp $  in $ \Sigma^+M'$.
This has the following important consequence.

\begin{elem}\label{form}
If $\,\omega$ is a $2$-form and $X$ is a vector field on $M'$ such that $\; \omega\cdot \Psi = X
\cdot \Psi \;$ holds, then
$$
\omega \;=\; (X \wedge \xi)^+ \;-\; \tfrac{1}{| \xi |^2}\, (X \wedge \xi)^- .
$$
where $\sigma^\pm$ denotes the self-dual and  anti-self-dual part of a $2$-form $\sigma$. 
In particular it follows that $B(\xi)= 0$ and that $X$ is orthogonal to $\xi$.
\end{elem}
\proof
Decomposing the generalized Killing spinor as $\Psi = \Psi^+ + \Psi^-$ and the $2$-form as
$\omega = \omega^+ + \omega^-$, the equation $X \cdot \Psi = \omega \cdot \Psi$ can be
rewritten as
\bea
\omega^- \cdot \Psi^++ \omega^+ \cdot \Psi^-  &= &X \cdot \Psi  \;=\; X \cdot (\Psi^+ + \Psi^-)
\;=\; X \cdot \xi \cdot \Psi^- \;-\; \tfrac{1\,}{| \xi |^2} \, X \cdot \xi \cdot \Psi^+\\
&=& (X \wedge \xi)^+ \cdot \Psi^- \;-\; g(X, \xi)\, \Psi^- \;-\; 
\tfrac{1\,}{| \xi |^2} \, (X\wedge \xi )^- \cdot \Psi^+\\
&&+\tfrac{1\,}{| \xi |^2} \, g(X, \xi)\,  \Psi^+ .
\eea
Comparing types, we find $\, \omega^+ = (X \wedge \xi)^+$ and 
$\, \omega^- = - \tfrac{1\,}{| \xi |^2} \, (X\wedge \xi )^- $. Moreover, since $\sigma^+ \cdot
\Psi^-$ is orthogonal to $\Psi^-$  for any $2$-form $\sigma$, the equation immediately implies $g(
X, \xi ) = 0$.
 Finally applying this result to Equation~\eqref{two-2} we obtain that
$g(B(Z), \xi)=0$ for any vector field $Z$, thus $B(\xi) = 0$.

\qed

\bigskip

Lemma~\ref{form} applied to Equations  \eqref{curv-2} and \eqref{two-2} allows us to express the
full curvature tensor of $(M, g) $ in terms of the endomorphism $A$. Indeed we immediately obtain
\beq\label{curv-3}
\tfrac12 \, g( \mathcal R(\sigma ), \tau )  \;+\; 2\, g( A(\sigma) , \tau )
\;=\; g(T(\sigma) \wedge \xi, \tau^+) \;\;-\;\; \tfrac{1\,}{ |\xi|^2} \,  g( T(\sigma) \wedge \xi,
\tau^-) ,
\eeq
for any $2$-forms $\sigma$ and $\tau$. Here  $\tau^+$ and $\tau^-$ are  self- and  anti-self dual 
part of $\tau$. The $T$-part needs a short calculation and is given in the following

\begin{elem}\label{T}
Let $\sigma$ and $\tau$ be any $2$-forms, then
$$
g( T(\sigma) \wedge \xi, \tau ) \;=\; 
2\, g(  B(\sigma^+(\xi)) , \tau (\xi))  \;-\; \tfrac{2\,}{ | \xi |^2} \, g(  B(\sigma^-(\xi)), \tau (\xi)).
$$
\end{elem}
\proof
Lemma~\ref{form} together with Equation~\eqref{two-2} imply for any vector field $Z$ the equation
$$
T^Z \;=\; 2\, (B(Z) \wedge \xi)^+ \;-\; \tfrac{2\,}{|\xi|^2} \, (B(Z) \wedge \xi)^- .
$$
Then $T(X  \wedge Y) = \sum_{i=1}^n T^{e_i}(X, Y)\, e_i = \sum_{i=1}^n g(   T^{e_i}, X \wedge Y)
e_i$. Thus replacing $X\wedge Y$ with any $2$-form $\sigma$ gives
\begin{align*}
T(\sigma ) & \;=\quad  2  \sum_{i=1}^n g( B(e_i) \wedge \xi, \sigma^+)\, e_i \;-\; \sum_{i=1}^n
\tfrac{2\,}{|\xi|^2} \,  g( B(e_i) \wedge \xi, \sigma^-)\,e_i \\
& \;= \; -  2 B(\sigma^+(\xi)) \;+\;  \tfrac{2\,}{|\xi|^2} \, B(\sigma^-(\xi)) .
\end{align*}
Taking the scalar product with $\tau(\xi) = \xi \lrcorner \,\tau$ proves the statement of the lemma.
\qed

\medskip

For later use we still need an expression in the special case $\sigma = \eta \wedge Y$, where $Y$ is an arbitrary  
vector field and $\eta$ is defined in \eqref{eta}.

\begin{ecor}\label{corT}
Let $Y$ be any vector field then
$$
T(\eta, Y) \;=\; (1-2h) \left(A^2(Y) -a A(Y) + \tfrac{\lambda}{4} \, Y \right) .
$$
\end{ecor}
\proof
In dimension $4$ we have $\ast (X\wedge Y) = - X \lrcorner \ast Y$ and 
$\ast (X \wedge Y)(\xi) =-  g( \ast Y, X \wedge \xi ) $. Hence,
taking $\sigma = X \wedge Y$, we deduce from the calculations in the proof of Lemma~\ref{T} that
$$
T(X \wedge Y ) \;=\; ( \tfrac{1\,}{|\xi|^2} - 1)  B((X \wedge Y)(\xi))
\;-\;
(1 +  \tfrac{1\,}{|\xi|^2}) B( \ast (X \wedge Y)(\xi)) .
$$
Thus specializing to $X \wedge Y = \eta \wedge Y$ the second summand vanishes. Recalling that
$\eta = - h \xi$, $|\xi|^2 = \tfrac{1-h}{h} $ and $B(\xi)=0$,  we get
$$
T(\eta, Y) \;=\; 
\tfrac{1 - |\xi|^2}{ | \xi |^2} g( \eta , \xi )B(Y) 
\;=\; (1-2h)B(Y) \;=\; (1-2h)(A^2(Y) - a A(Y)+ \tfrac{\lambda}{4} \, Y  ) .
$$
\qed

\medskip

In dimension $4$ the Einstein condition is equivalent to having
$\mathcal R : \Lambda^2_\pm T \rightarrow \Lambda^2_\pm T$, i.e. the curvature operator preserves
the space of self-dual and
anti-self-dual forms. In particular we have $g( \mathcal R (\sigma^+), \tau^-) = 0$.  Let $e_1:=
\tfrac{\xi }{ \|  \xi \|}$.
Substituting
$\sigma = \sigma^+$ and $\tau = \tau^-$ in  \eqref{curv-3} and using Lemma~\ref{T} yields
\beq\label{AB}
0 \;=\; g( A(\sigma^+), \, \tau^-  ) \;+\; g(  B(\sigma^+(e_1)), \, \tau^-(e_1) ).
\eeq

\medskip

We will use \eqref{AB}
to show that $A(e_1) = a_1 e_1$ for some real function $a_1$. Indeed the condition $B(e_1) = 0$ implies 
$
A^2(e_1) - a A(e_1) + \tfrac{\lambda}{4}e_1 = 0 .
$
Thus the space $\mathrm{span} \{e_1, A(e_1)\}$ is invariant under $A$ and we may choose a local orthonormal
frame $\{e_1, e_2,e_3, e_4\}$, with $e_1:= \tfrac{\xi\, }{ \|  \xi \|^2}$ and
$$
A e_1 = a_1 e_1 + a_{12} e_2, \quad A e_2 = a_{12} e_1 + a_2 e_2,\quad A e_3 = a_3 e_3, \quad A e_4
= a_4 e_4 .
$$
Consider the $2$-forms $ \sigma^+ = e_1 \wedge e_3 - e_2 \wedge e_4$ and $\tau^- = e_1 \wedge e_4 -
e_2 \wedge e_3$. Then
\beq\label{ev}
0 \;=\; g( A(\sigma^+), \, \tau^- ) \;=\; - a_{12} \, (a_4 \,+\, a_3 ) .
\eeq
Next consider the $2$-forms $\sigma^+ = e_1 \wedge e_4 + e_2\wedge e_3$ and $\tau^- e_1 \wedge e_4
- e_2 \wedge e_3$. Then
\bea
0 &  =&  g( A(\sigma^+), \, \tau^- )  \;+\; g(  B(e_4), e_4 )\;=\; (a_1a_4 \,-\,a_2a_3)\;  +\;
(a_4^2 \,-\,a a_4 \,+\, \tfrac{\lambda}{4} ) \\
&  = & -a_2(a_3 \,+\,a_4) \;-\; a_3a_4 \;+\; \tfrac{\lambda}{4}
\eea
If $a_3 + a_4 = 0$, then $a_3 a_4 = \tfrac{\lambda}{4}>0$, which is impossible. Thus $a_3 + a_4
\neq 0$ and $a_{12}$ has to vanish, because of \eqref{ev}. Consequently, around every point in $M'$
we have a local orthonormal frame $e_1 := \tfrac{\xi }{ \|  \xi \|} , e_2, e_3, e_4$ with 
$Ae_i = a_i e_i$ and such that the eigenvalues $a_i$ satisfy the relation 
$
a_2a_3 + a_2 a_4 + a_3a_4 = \tfrac{\lambda}{4}
$.
Moreover $a_1^2 - aa_1 + \tfrac{\lambda}{4} = 0$ and in particular, since $\lambda >0$, the function $a_1$
is nowhere zero. 

\medskip

Using Lemma~\ref{dh} (i) and (iii) we get $0 = d(A(\eta)=d(a_1\eta)=da_1 \wedge \eta$ and thus $d
a_1$ is collinear to $\eta$. The precise relation is given in the following

\begin{epr}\label{da}
$
\quad d a_1 \;=\; \tfrac{1 - 2h}{h(1-h)} (\tfrac{\lambda}{4} - 3a_1^2) \, \eta
$
\end{epr}
\proof
Let $f$ be a function with $da_1 = f \eta$. Then
$
da_1(\eta) = f |\eta|^2 = f h (1-h)
$
and
$
f = \tfrac{\eta(a_1)}{h(1-h)}
$.
In order to compute $\eta(a_1)$, we take the covariant derivative of $A(\eta) = a_1 \eta$
in direction of the vector field $Y$. Using Lemma~\ref{dh} (ii) we get
\bea
(\nabla_Y A)\, \eta  & =&- \,  A(\nabla_Y\eta)\;+\; Y(a_1)\,  \eta \;+\; a_1(1-2h)\, A(Y) \\
& = & - (1-2h) A^2(Y)\;+\; Y(a_1)\, \eta \;+\; a_1(1-2h) A(Y) .
\eea
Next we apply \eqref{da1} and Corollary~\ref{corT} to interchange $Y$ and $\eta$.
We obtain
\bea
(\nabla_\eta A)\, Y & =& T(\eta, Y)\; -\;  (1-2h) A^2(Y)\;+\; Y(a_1)\, \eta \;+\; a_1(1-2h) A(Y) \\
& = &
Y(a_1) \;+\; (1-2h) \left( (a_1 -a ) A(Y) \,+\, \tfrac{\lambda}{4} \, Y \right) .
\eea
Since $|A|^2 = \tr (A^2)$ and $\scal $ is constant, \eqref{three2} implies
$
\eta(|A|^2) = \eta (\tr(A^2) ) = \eta (a^2) = 2a \eta(a)
$.
On the other hand, computing $\eta(|A|^2) $ with
$
\eta(|A|^2) = \nabla_\eta |A|^2 = 2 g( \nabla_\eta A, A) = 2 g( (\nabla_\eta A)e_i, A(e_i) )
$
gives
\bea
a \, \eta(a) & = &A(\eta) (a_1) \;+\; (1-2h) \left( (a_1 -a ) \tr(A^2) \,+\, \tfrac{\lambda}{4} \, a \right) \\
& = & a_1 \, \eta(a_1)  \;+\; (1-2h) \left( (a_1 -a ) (a^2 - \lambda ) \,+\, \tfrac{\lambda}{4} \, a \right).
\eea
From $B(\xi)=0$ we have $a_1^2 - a a_1 + \tfrac{\lambda}{4} = 0 $ and thus
$
a \, \eta(a) - a_1 \eta(a_1) = - \tfrac{(a_1-a)^2}{a_1} \eta(a_1)
$.
Another simple calculation gives
$
 (a_1 -a ) (a^2 - \lambda ) \,+\, \tfrac{\lambda}{4} \, a = (a_1-a)^2(4a_1-a)
$.
Substituting this into the equation above yields
$$
\eta (a_1) \;=\; (1-2h)(aa_1 - 4a_1^2) \;=\; (1-2h)(\tfrac{\lambda}{4} - 3a_1^2) .
$$
\qed

\medskip

Comparing the expression of $da_1$  given in Proposition~\ref{da} and the expression of $dh$ given
in Lemma~\ref{dh} (ii) we get
$$d a_1 \;=\; \tfrac{1 - 2h}{2h(1-h)} \left(\tfrac{\lambda}{4a_1} - 3a_1\right) dh.$$
This shows that the function 
\beq\label{c1}
\; C := (h(1 - h))^3 \, \left(a_1^2 - \tfrac{\lambda}{12}\right) \;
\eeq
is constant on $M'$. Note that although the function $a_1$ is only defined on $M'$, the function
$l:=h(1-h)a_1^2$ is well-defined on the whole $M$.
Indeed, from Lemma~\ref{dh} (i) and \eqref{et} we get $|dh|^2=4a_1^2h(1-h)=4l$. Using the density
of $M'$ in $M$, \eqref{c1} shows that
\beq\label{c2}
\; C := (h(1 - h))^2 \, \left(l - \tfrac{h(1-h)\lambda}{12}\right) \;
\eeq
is constant on $M$. Moreover this constant turns out to be zero because of

\begin{elem}
The function $h(1-h)$ has a zero and in particular the constant $C$ vanishes.
\end{elem}
\proof
Since $M$ is compact $h$ attains its absolute minimum at some point $x_0 \in M$. 
By \eqref{et} $h$ takes values in $[0,1]$. Clearly $h(x_0)<1$, since otherwise $h\equiv 1$ on $M$,
i.e. $\Psi^-\equiv 0$, which is impossible.

Assume that $h(x_0)\ne 0$.
Then $x_0\in M'$ so Lemma~\ref{dh} (i) together with $dh_{x_0} = 0$ give $2a_1(x_0)\eta_{x_0} = 0$.
 As $a_1$ is nowhere zero on $M'$, it
follows $\eta_{x_0}=0$ and thus $ 0 = |\eta|^2(x_0) = h(1-h)(x_0)$, i.e. $h(x_0)=0$. This
contradiction shows that actually $h$ vanishes at $x_0$.
\qed

\bigskip
We can now conclude:

\begin{ath}\label{dim4}
Let $(M, g)$ be a compact $4$-dimensional Einstein manifold of positive scalar curvature, admitting
a generalized Killing spinor $\Psi$. Then $(M, g)$ is
isometric to the standard sphere and $\Psi$ is an ordinary Killing spinor.
\end{ath}
\proof
Since $C=0$ and $h$ is non-constant, \eqref{c2} gives $a_1^2 = \tfrac{\lambda}{12}$, so $a^2 =
\tfrac43\lambda$. In particular, the function $a$
is constant on $M'$, and thus on $M$. From Corollary \ref{Delta} it follows that $1-2h$ is a an
eigenfunction for the
Laplace operator for the eigenvalue $4(a^2-\lambda) = \tfrac{4}{3} \lambda$. According to the
Lichnerowicz-Obata Theorem,
this is the lowest possible eigenvalue of the Laplace operator on compact Einstein manifolds, and
it characterizes the round sphere.

Moreover,  from Equation \eqref{three2} we obtain 
$$\tr\left(A-\tfrac a4\id\right)^2=\tr(A^2)-\tfrac{a^2}4=(a^2-\lambda)-\tfrac{a^2}4=0.
$$
This shows that the symmetric endomorphism $A$ is a constant multiple of the identity map, so
$\Psi$ is a Killing spinor and thus the manifold is isometric to $\SM^4$, cf. \cite{ba,bfgk,hi86}.
\qed

\medskip

\subsection{The case of dimension $5$}

We start by reviewing the algebraic theory of spinors in dimension 5. Let $M$ be a 5-dimensional
spin manifold. Since the spin representation is isomorphic to the standard representation of
$\Spin(5)=\Sp(2)$ on $\RM^8=\HM^2$, the spin bundle $\Sigma M$ carries a quaternionic structure.
However, as the Clifford algebra $\Cl_5$ is isomorphic with $\CM(4)$, only one of the complex
structures on $\Sigma M$  (the one given by Clifford multiplication with the volume element)
commutes with the Clifford product with vectors. Let us call this complex structure $I$ and denote
by $J$ and $K$ two other complex structures on $\Sigma M$ orthogonal to $I$ and anti-commuting. It
is easy to check that $J$ and $K$ anti-commute with the Clifford product with vectors. 

For every nowhere vanishing spinor $\Psi$, the spin bundle has the following orthogonal direct sum decomposition:
\beq\label{deco} \Sigma M=\T M\.\Psi\oplus \la\Psi\ra\oplus \la J\Psi\ra\oplus  \la K\Psi\ra.
\eeq
Indeed, it is straightforward to check from the above properties of the complex structures $J$ and
$K$ that all factors are mutually orthogonal.
Since $I\Psi$ is orthogonal to the last three factors, we must have $I\Psi\in \T M\.\Psi$, so there
exists a unit vector field $\xi$ such that 
\beq\label{xi}I\Psi=\xi\.\Psi.
\eeq
We denote by $D:=\xi^\perp$ the distribution orthogonal to $\xi$. For every $X\in D$ the spinor
$X\.I\Psi$ is orthogonal to $\Psi$, $I\Psi$,  $J\Psi$ and  $K\Psi$, so there exists a vector $Y_X\in
D$ such that $X\.I\Psi=Y_X\.\Psi$. We denote by $L$ the endomorphism of $\T M$ which maps $\xi$ to
$0$ and $X$ to $Y_X$ for $X\in D$. It is easy to check that $L$ is skew-symmetric and satisfies the
relations 
\beq\label{t1} L^2=-\id+\xi\otimes\xi,
\eeq
(so $L$ defines a complex structure on $D$), and
\beq\label{t2} X\.I\Psi=LX\.\Psi-g(X,\xi)\Psi,\qquad\forall X\in \T M.
\eeq
This last relation allows us to explicit the Clifford product of 2-forms with $\Psi$. We decompose
$\Lambda^2M$ as
$$\Lambda^2M=\xi\wedge \T M\oplus \Lambda^{(1,1)}D\oplus \Lambda^{(2,0)+(0,2)}D$$
and using \eqref{t2} we get
\beq\label{cl1} (\xi \wedge X)\.\Psi=-LX\.\Psi,\qquad\forall X\in \T M,
\eeq
\beq\label{cl2}  \Lambda^{(1,1)}D\.\Psi=\la I\Psi\ra=\la \xi\.\Psi\ra,
\eeq
\beq\label{cl3} \Lambda^{(2,0)+(0,2)}D\.\Psi=\la J\Psi\ra\oplus  \la K\Psi\ra.
\eeq
The last relation actually yields a trivialization of $\Lambda^{(2,0)+(0,2)}D$ (and thus a $\SU(2)$
reduction of the structure group of $M$), but we will not need this in the sequel. 

We are now ready to prove the main result of this section:

\begin{ath}\label{dim5}
Any generalized Killing spinor on the $5$-sphere $\SM^5$ is a real Killing spinor.
\end{ath}
\proof
Let $\Psi$ be a generalized Killing spinor on $\SM^5$ satisfying \eqref{gks}. The curvature
endomorphism of the sphere is minus the identity on $2$-forms, so \eqref{curv-2} reads 
$$T(X\wedge Y)\.\Psi=\left[2AX\wedge AY-\tfrac12 X\wedge Y\right]\.\Psi,\qquad\forall X,Y\in \T M.$$
From \eqref{cl1}--\eqref{cl3} we deduce that 
\beq\label{11}AX\wedge AY-\tfrac14 X\wedge Y\in\xi\wedge\T M\oplus \Lambda^{(1,1)}D,\qquad\forall
X,Y\in \T M.
\eeq 
Let $A\xi=\alpha \xi+\zeta$ with $\zeta \in D$.
Taking $X=\xi$ in \eqref{11} we obtain that $\zeta\wedge AY$ belongs to $\Lambda^{(1,1)}D$ for
every $Y$ orthogonal to $A\xi$. Assume that 
$\zeta\ne0$. Then $AY$ belongs to the subspace spanned by $\zeta$ and $L\zeta$ for all $Y$
orthogonal to $A\xi$, whence the image of $A$ is a subset of 
$\la\xi,\zeta,L\zeta\ra$. Since $A$ is symmetric, $\Ker(A)$ contains the orthogonal complement of
$\la\xi,\zeta,L\zeta\ra$ in $\T M$. Let $Y$ be a vector in this orthogonal complement. By \eqref{11}
again, $X\wedge Y \in\Lambda^{(1,1)}D $ for all $X\in D$, which is a contradiction for $X=\zeta$.
This shows that 
$\zeta=0$, so $D$ is left invariant by $A$. 

Let $\{e_1,e_2,e_3,e_4\}$ be an orthonormal basis of eigenvectors of the restriction of $A$ to $D$
corresponding to the eigenvalues $\alpha_i$. Thus in the basis $\{\xi,e_1,e_2,e_3,e_4\}$ the matrix
of $A$ is diagonal, with entries $\{\a,\a_1,\a_2,\a_3,\a_4\}$. 
Equation \eqref{11} implies that 
$$(\a_i\a_j-\tfrac14)\,e_i\wedge e_j\in \Lambda^{(1,1)}D$$ for all subscripts $1\le i,j\le 4.$
This shows that for every subscript $i$, there are at least two other subscripts $j$ and $k$ such
that $\a_i\a_j=\a_i\a_k=\tfrac14$. Up to a permutation we can thus assume that $\a_1=\a_2$ and
$\a_3=\a_4=\tfrac1{4\a_1}$. We see that either $\a_1=\a_2=\a_3=\a_4$, or $\a_1\ne \a_3$, in which
case $e_1\wedge e_2$ and $e_3\wedge e_4$ belong to $\Lambda^{(1,1)}D$. In this last case $L$
preserves the eigenspaces $\la e_1,e_2\ra$ and $\la e_3,e_4\ra$ of $A$, so in both cases $L$ and $A$
commute. 

We now take a covariant derivative with respect to an arbitrary vector $X$ in \eqref{xi} and use
\eqref{t2} to obtain 
\bea \nabla_X\xi\.\Psi&=&\nabla_X(\xi\.\Psi)-\xi\.\nabla_X\Psi\\
&=&I(AX\.\Psi)-\xi\.AX\.\Psi=AX\.I\Psi+AX\.\xi\.\Psi+2g(AX,\xi)\Psi\\
&=&2AX\.I\Psi+2g(AX,\xi)\Psi=2LAX\.\Psi,
\eea
whence 
\beq\label{nxi}\nabla_X\xi=2LAX,\qquad\forall X\in \T M. \eeq
Since $L$ and $A$ commute, $LA$ is skew-symmetric, thus $\xi$ is a Killing vector field on $\SM^5$,
i.e. there exists a skew-symmetric matrix $M\in \so(6)$ such that $\xi_x=Mx$ for every
$x\in\SM^5\subset \RM^6$. Moreover, $\xi$ has constant length 1, thus $M$ is orthogonal, so
$M^2=-\id$. On the other hand, the covariant derivative of $\xi$ can also be computed by projecting
in $\T \SM^5$ the Euclidean covariant derivative in $\RM^6$. We thus obtain
$(\nabla_X\xi)_x=\mathrm{pr}_{T_x\SM^5}(MX)=MX-\la x,MX\ra x$. Let $\f$ denote the skew-symmetric
endomorphism corresponding to $\nabla\xi$. The previous relation implies
$$\f^2(X)=\f(MX-\la x,MX\ra x)=M^2X-\la x,MX\ra Mx=-X+\la Mx,X\ra Mx=-X+g(X,\xi)\xi.$$
(Note that this relation could also have been obtained by saying that every unit Killing vector
field on $\SM^5$ is Sasakian). Comparing it with \eqref{nxi} and using \eqref{t1} yields
$$-X+g(X,\xi)\xi=4L^2A^2X=-4A^2X+4g(A^2X,\xi)\xi.$$
Consequently, the restriction of $4A^2$ to $D$ is the identity, thus $\tr A^2=\alpha^2+1$. Moreover
the eigenvalues $\a_i$ of $A|_D$ belong to $\{\pm\tfrac12\}$. Since these eigenvalues are pairwise
equal, assuming that they are not all equal then $\tr(A|_D)=0$ so $a=\tr A=\alpha$. This contradicts
Equation \eqref{three2} which in our case reads $a^2-\tr A^2=\tfrac14\scal=5$. In the remaining
cases the eigenvalues of $A|_D$ are all equal to either 
$\a_1=\tfrac12$ or $\a_1=-\tfrac12$, so $a=\alpha+4\a_1$ and from \eqref{three2} again we get
$(\alpha+4\a_1)^2-(1+\a^2)=5$, which finally gives 
$\alpha\a_1=\tfrac14$ i.e. $\alpha=\alpha_1$ and $A$ is a constant scalar matrix $\pm\tfrac12\id$.
This finishes the proof of the theorem. \qed

\subsection{The case of dimension $6$ and $7$}\label{4.4} As already
mentioned, generalized Killing spinors are equivalent to half-flat  $\SU(3)$-structures
\cite{chs,hi03} in dimension $6$ and to co-calibrated $\mathrm G_2$-structures \cite{cs06,fg}
in dimension $7$. 

Using this correspondence, examples of Einstein metrics with generalized Killing spinors in these
dimensions can be found in the recent litterature. In \cite{sh}, Table 3, p. 74, Schulte-Hengesbach constructs a half-flat 
structure on  the Riemannian product $\SM^3 \times \SM^3$ (see also \cite{sh}, Remark 1.12, p. 87).
In fact Schulte-Hengesbach classifies all half-flat structures on the product of two 3-dimensional spheres, and it turns out that for a 
certain choice of the parameters the corresponding structure is compatible with the product metric $\SM^3\times \SM^3$
(see also  \cite{ms}, p. 14).

The Fubini-Study metric on $\CM \rm P^3$ admits a half-flat structure with respect to the
non-integrable almost complex structure.  This example can be found in  \cite{conti}, Section 4.5,
Prop. 4.12. Conti considers the complex projective space  $\CM \mathrm P^3$ realized as a
hypersurface in the total space of the vector bundle of anti-self-dual $2$-forms 
$\Lambda^2_- \SM^4 $ equipped with the parallel $\mathrm G_2$-structure found by Bryant and Salamon
\cite{bs}.

In fact, using the methods of \cite{bfgk}, Chapter 5.4,  it is easy to show that on the homogeneous
spaces  $\SM^3\times \SM^3$, $\CM \rm P^3$ and the flag 
manifold $\SU(3)/\TM^2$ there exists a $1$-parameter family of metrics with a generalized Killing
spinor, given as a constant map on the respective groups. For each of the three cases, this family
of metrics contains exactly two Einstein metrics. One of these Einstein metrics is compatible with
a nearly K\"ahler structure and the corresponding spinor is a Killing spinor. The second Einstein
metric is the standard K\"ahler-Einstein metric on the two twistor spaces $\CM \rm
P^3$ and $\SU(3)/\TM^2$ and the product metric on $\SM^3\times \SM^3$. Note that in the first two
cases, the generalized Killing spinor turns out to be a K\"ahlerian Killing spinor \cite{ki86,a95}.

In dimension 7, examples of generalized Killing spinors which are not Killing were recently
constructed by Agricola and Friedrich \cite{af10} on every 3-Sasakian metric (it is well known that
$3$-Sasakian metrics are automatically Einstein). They actually show that every 7-dimensional
3-Sasakian manifold carries a so-called canonical spinor $\Psi_0$, which is not a Killing spinor
itself, but generates the 3-dimensional space of Killing spinors in the sense that the three Killing
spinors of the $3$-Sasakian structure are obtained by Clifford product of $\Psi_0$ with the three
Sasakian Killing vector fields. Agricola and Friedrich show that the $\G_2$-structure defined by
$\Psi_0$ is co-calibrated not only on the original metric $g$ but on the whole 1-parameter family
of metrics $g_t$ obtained by rescaling  $g$ along the 3-dimensional Sasakian distribution. Thus
$\Psi_0$ is a generalized Killing spinor for each metric in this family, which contains two Einstein
metrics: the 
$3$-Sasakian metric $g$ for $t=1$ and a proper $\mathrm G_2$-metric for $t=\tfrac15$ (cf.
\cite{fkms}). For this second Einstein metric, $\Psi_0$ turns out to be a genuine Killing spinor. 

Note in particular that this construction gives a generalized Killing spinor which is not Killing
on the standard sphere $\SM^7$.

\section{Final remarks and open questions}

In view of the correspondence between generalized Killing spinors and hypersurface embeddings in
manifolds with parallel spinors, we obtain the following corollaries of our main results:

\begin{ecor}
Let $(M^4,g)$ be a compact Einstein hypersurface with positive scalar curvature in a Riemannian
product $(\Z^{5},g^\Z)=\RM\times (N^4,h)$ where $(N,h)$ is simply connected and hyperk\"ahler (e.g.
the flat space $\RM^4$, a $\mathrm{K}3$ surface, an Eguchi-Hanson manifold, a Taub-NUT or a
Kronheimer ALE space). Then $(N,h)$ is flat and $(M^4,g)=\SM^4$.
\end{ecor}

\begin{proof}
The manifold $(\Z,g^\Z)$ is spin and has a parallel spinor. By \cite{bgm}, Equation (30), $(M^4,g)$
carries a generalized Killing spinor, so $(M^4,g)=\SM^4$ by Theorem \ref{dim4}. From the uniqueness
part of Theorem 1.1 in \cite{amm} we deduce that the ambient metric $g^\Z$ is flat. 
\end{proof}

A similar argument together with Theorem \ref{dim5} yields:

\begin{ecor}
Let $(\Z^{6},g^\Z)$ be a simply connected Ricci-flat K\"ahler threefold (e.g. the flat space
$\RM^6$, a Calabi-Yau threefold, or 
$\RM^2\times \mathrm{K}3$). If $\SM^5$ has an isometric embedding in $(\Z,g^\Z)$, then $\Z$ is a
flat space.
\end{ecor}

We finally list some open questions which arose during the preparation of this work, which we think
worth of further investigation:

\begin{itemize}
\item Does the sphere $\SM^n$ carry generalized Killing spinors which are not Killing? From the
above, we know that the answer is ``yes'' for $n=3$ and $n=7$, ``no'' for $n=2$, $n=4$ and $n=5$, and
unknown for $n=6$ and $n\ge 8$. It is surprising that one does not know whether a half-flat
structure on $\SM^6$ is necessarily nearly K\"ahler.
\item Find all generalized Killing spinors on a given spin manifold. To our knowledge, the only
cases where a complete answer is available are given by Theorems \ref{dim4} and \ref{dim5}. Even
for one of the simplest possible manifolds, the round sphere $\SM^3$, the set of generalized Killing
spinors is unknown.
\item In the real analytic case, a spin manifold with generalized Killing spinors embeds as
a hypersurface in a manifold with parallel spinors (\cite{amm}, Theorem 1.1). What is (locally) the
ambient metric corresponding to the examples above, e.g. for the spinors on $\SM^3$ constructed in
Section \ref{4.1}? This metric is interesting since it is hyperk\"ahler, non-flat, and contains
round spheres as hypersurfaces.
\item In all available examples of generalized Killing spinors on Einstein manifolds, the
symmetric tensor $A$ has constant eigenvalues. Is this a general phenomenon?
\item Is it possible to construct examples of generalized Killing spinors on 3-Sasakian manifolds
of dimension $4n+3\ge 11$
using methods similar to those in \cite{af10}?
\end{itemize}


\begin{thebibliography}{5}

\bibitem{af10} I.~Agricola, Th.~Friedrich,
{\sl 3-Sasakian manifolds in dimension seven, their spinors and $\G_2$-structures},
 J.\ Geom.\ Phys.\ {\bf 60} (2010), no.\ 2, 326--332.


\bibitem{amm} B.~Ammann, A.~Moroianu, S.~Moroianu,
{\sl The Cauchy problems for Einstein metrics and parallel spinors}, 
Commun.\ Math.\ Phys.\ (2013), doi: 10.1007/s00220-013-1714-1.


\bibitem{ba}
C.~B{\"a}r,
 {\sl Real {K}illing spinors and holonomy}, Commun.\ Math.\ Phys.\
  \textbf{154} (1993), 509--521.

\bibitem{bgm}
{C.~B\"ar, P.~Gauduchon, A.~Moroianu,}
{\sl Generalized Cylinders in Semi-Riemannian and Spin Geometry, }
Math.\ Z.\ {\bf 249} (2005), 545--580.

\bibitem{bfgk} 
{ H.~Baum, Th.~Friedrich, R.~Grunewald, I.~Kath, }
{\it Twistor and Killing Spinors on Riemannian Manifolds, }
Teubner-Verlag, Stuttgart-Leipzig, 1991.

\bibitem{br}
R.L.~Bryant,
{\sl Non-embedding and non-extension results in special holonomy},
The many facets of geometry, 346--367, Oxford Univ.\ Press, Oxford, 2010. 

\bibitem{bs}
R.L.~Bryant, S.~Salamon,
{\sl On the construction of some complete metrics
with exceptional holonomy},
Duke Math. J.,  {\bf 58} (1989), 829--850.

\bibitem{chs}
S.~Chiossi, S.~Salamon,
{\sl The intrinsic torsion of $\SU(3)$ and $\G_2$ structures},
 Differential geometry, Valencia, 2001, 115--133, World Sci.\ Publ., River Edge, NJ, 2002. 

\bibitem{conti} D.~Conti,
{\sl Special holonomy and hypersurfaces},
thesis, Scuola Normale Superiore di Pisa, 2005.


\bibitem{cs06} D.~Conti, S.~Salamon, {\sl Reduced holonomy, hypersurfaces and extensions},
Int.\ J.\ Geom.\ Methods Mod.\ Phys.\ {\bf 3} (2006), no.~5-6, 899--912.

\bibitem{cs07} D.~Conti, S.~Salamon, {\sl Generalized Killing spinors in dimension 5}, 
Trans.\ Amer.\ Math.\ Soc.\ {\bf 359} (2007), no.\ 11, 5319--5343.

\bibitem{fg}
M.~Fernandez, A.~Gray,
{\sl Riemannian manifolds with structure group $\G_2$},
 Ann. Mat. Pura Appl. (4) {\bf 132} (1982), 19--45.

\bibitem{fkms} Th. Friedrich, I. Kath, A. Moroianu, U. Semmelmann, {\sl On nearly parallel
$G_2$-structures, } 
J.\ Geom.\ Phys.\ {\bf
  23} (1997), 259--286.

\bibitem{friedrich:98}
Th.~Friedrich, {\sl On the spinor representation of surfaces in Euclidean
  $3$-space}, J.\ Geom.\ Phys.\ \textbf{28} (1998), 143--157.

\bibitem{kf00}
Th.~Friedrich, E.C.~Kim, {\sl The Einstein-Dirac equation on Riemannian spin manifolds,}
J.\ Geom.\ Phys.\ {\bf 33} (2000), no.\ 1-2, 128--172. 

\bibitem{kf01}
Th.~Friedrich, E.C.~Kim,
{\sl  Some remarks on the Hijazi inequality and generalizations of the Killing equation for spinors}
J.\ Geom.\ Phys.\  {\bf 37}  (2001),  no.\ 1-2, 1--14. 

\bibitem{hi86}
O.~Hijazi,
{\sl 
Caract\'erisation de la sph\`ere par les premi\`eres valeurs propres de l'op\'erateur de Dirac en
dimensions 3, 4, 7 et 8},
C.\ R.\ Acad.\ Sci.\ Paris S\'er.\ I Math.\ {\bf 303} (1986), no.\ 9, 417--419.


\bibitem{hi03} N.~Hitchin, {\sl Stable forms and special metrics}, 
Global differential geometry: the mathematical legacy of Alfred Gray (Bilbao, 2000), 70--89,
Contemp.\ Math.\, {\bf 288}, Amer.\ Math.\ Soc.\, Providence, RI, 2001. 

\bibitem{ki86} K.-D.\ Kirchberg, {\sl An estimation for the first eigenvalue of the Dirac operator
on closed K\"ahler manifolds of positive scalar curvature,} Ann.\ Global Anal.\ Geom.\ {\bf 4}
(1986), no.\ 3, 291--325.

\bibitem{ku} W.\ K\"uhnel,
{\it Differential geometry. Curves� surfaces� manifolds}, Student Mathematical Library, {\bf 16}.
American Mathematical Society, Providence, RI, 2002.

\bibitem{lm} B.~Lawson, M.-L.~Michelson, {\it Spin Geometry}, Princeton University
Press, Princeton 1989.

\bibitem{ms}
T.B.~Madsen, S.~Salamon,
{\sl Half-flat structures on $S^3\times S^3$},
{\tt math.DG/1211.6845}.

\bibitem{morel03}
 B.~Morel,
{\sl The energy-momentum tensor as a second fundamental form},
  {\tt math.DG/0302205}.

\bibitem{a95} A.\ Moroianu, {\sl La premi{\`e}re valeur propre de  
l'op{\'e}rateur de Dirac sur les vari{\'e}t{\'e}s k{\"a}hl{\'e}rien\-nes  
compactes, } Commun.\ Math.\ Phys.\ {\bf 169} (1995), 373--384.

\bibitem{sh}
F.~Schulte-Hengesbach,
{\sl Half-flat structures on Lie groups},
thesis, Universit\"at Hamburg, 2010,\phantom{xxxxxx}
http://www.math.uni-hamburg.de/home/schulte-hengesbach/diss.pdf.

\bibitem{seki} K. Sekigawa, {\sl On some compact Einstein almost K{\"a}hler
manifolds}, J.\ Math.\ Soc.\ Japan {\bf 39} (1987), 677--684.



\end{thebibliography}
\end{document}